\documentclass[a4paper, 10pt]{article}

\usepackage{graphics,graphicx}
\usepackage{amssymb,amsmath}
\usepackage{color}
\usepackage{enumitem}

\newtheorem{theorem}{Theorem}[section]
\newtheorem{problem}[theorem]{Problem}
\newtheorem{claim}{Claim}

\newtheorem{lemma}[theorem]{Lemma}

\def \no {\noindent}
 
 \def \sm {\setminus}
 \def \es {\emptyset}
 \def \sset {\sqsubseteq}

\newenvironment{proof}[1][]%
{\noindent {\setcounter{equation}{0}\it Proof.
}{#1}{}}{\hfill$\Box$\vspace{2ex}}

\def\longbox#1{\parbox{0.85\textwidth}{#1}}

\begin{document}

\title{Polynomial Cases for the Vertex Coloring Problem}

\author{T. Karthick \thanks{Computer Science Unit, Indian Statistical
Institute, Chennai Centre, Chennai 600029, India.}
\and%
Fr\'ed\'eric Maffray\thanks{CNRS, Laboratoire G-SCOP, Univ. Grenoble, France.}
\and%
Lucas Pastor\thanks{Laboratoire G-SCOP, Univ. Grenoble, France.}}

\maketitle

\begin{abstract}
The computational complexity of the Vertex Coloring problem is known
for all hereditary classes of graphs defined by forbidding two
connected five-vertex induced subgraphs, except for seven cases.  We
prove the polynomial-time solvability of four of these problems: for
($P_5$, dart)-free graphs, ($P_5$, banner)-free graphs, ($P_5$,
bull)-free graphs, and (fork, bull)-free graphs.
\end{abstract}


\section{Introduction}

Vertex coloring of graphs is one among the basic graph colorings and
has a long history starting with the four color problem, and is widely
studied in graph theory and in theoretical computer science.  It
occupies a central place in the complexity theory of algorithms and
arises naturally in many real world applications such as storage
problem, register allocation and time table scheduling.  Recent
publications show that vertex coloring problems still receive maximum
attention.

A \emph{vertex coloring} (or simply \emph{coloring}) of a graph $G$ is
an assignment of colors to the vertices of $G$ such that no two
adjacent vertices receive the same color.  That is, a partitioning of
the vertex set of $G$ into stable sets (called \emph{color classes}),
where a \emph{stable set} is a set of pairwise nonadjacent vertices.
The minimum number of colors required to color $G$ is called the
\emph{chromatic number} of $G$, and is denoted by $\chi(G)$.  Given a
graph $G$, the \textsc{Minimum Vertex Coloring} (VC) problem is to
determine the chromatic number $\chi(G)$.  The VC problem is well
known to be $NP$-complete in general, see \cite{GJ}, and also in may
restricted classes of graphs.  Lund and Yannakakis \cite{LY} showed
that there exists a constant $\epsilon > 0$ such that approximating
the chromatic number of an arbitrary graph within a factor of
$n^\epsilon$ is $NP$-hard, where $n$ is the number of vertices.  This
result is further improved by Feige and Kilian \cite{FK}, who proved
that the chromatic number cannot be approximated within a factor of
$n^{1-\epsilon}$, for any $\epsilon > 0$, unless $NP \neq ZPP$.  These
algorithmic issues are main motivations for current research to study
the VC problem in restricted classes of graphs.

A class of graphs $\cal{X}$ is {\it hereditary} if every induced
subgraph of a member of $\cal{X}$ is also in $\cal{X}$.  If $\cal{F}$
is a family of graphs, a graph $G$ is said to be {\it $\cal{F}$-free}
if it contains no induced subgraph isomorphic to any graph in
$\cal{F}$.  In this paper, we are interested in the VC problem for
some hereditary classes of graphs which are defined by two forbidden
induced subgraphs.  The VC problem remains $NP$-complete even for
restricted classes of graphs, such as triangle-free graphs \cite{MP},
$P_6$-free graphs \cite{Huang}, and $K_{1, 3}$-free graphs (see
\cite{LM}).  But, for many classes of graphs, such as perfect graphs
\cite{GLS} and for ($2P_3$, triangle)-free graphs
\cite{2P3-triangle-free}, the VC problem can be solved in polynomial
time.  The VC problem for $P_5$-free graphs is $NP$-complete
\cite{Kral}, but for every fixed $k$, the problem of coloring a
$P_5$-free graph with $k$ colors admits a polynomial-time algorithm
\cite{HKLSS}.  Kr\'al et al.~\cite{Kral} showed that the VC problem is
solvable in polynomial time for $H$-free graphs, whenever $H$ is a
(not necessarily proper) induced subgraph of $P_4$ or $P_3 +K_1$;
otherwise, the problem is $NP$-complete.  When we forbid two induced
subgraphs, only partial results are known for the VC problem.  The
motivation of this paper is the following open problem of Golovach et
al.  \cite{comp-survey}.

\begin{problem}[\cite{comp-survey}]
Complete the classification of the complexity of the \textsc{Vertex
Coloring} problem for $(G_1, G_2)$-free graphs.
\end{problem}

We refer to \cite{comp-survey} for a recent comprehensive survey and
for other open problems on the computational complexity of the VC
problem for classes of graphs defined by forbidden induced subgraphs,
and we refer to Theorem~1 of \cite{ML}, for the complexity dichotomy
of the VC problem for some classes of graphs which are defined by two
forbidden induced subgraphs.  It is known that the VC problem is
either $NP$-complete or polynomial time solvable for the classes of
graphs which are defined by two four-vertex forbidden induced
subgraphs, except for three classes of graphs namely, $(O_4,
C_4)$-free graphs, $(K_{1, 3}, O_4)$-free graphs, and for $(K_{1, 3},
K_2+O_2)$-free graphs.  Recently, the complexity dichotomy of the VC
problem for classes of graphs which are defined by two connected
five-vertex forbidden induced subgraphs received considerable
attention.  The VC problem is known to be solvable in polynomial time
for: $(P_5$, gem)-free graphs \cite{BBKRS}, $(P_5,
\overline{P_5})$-free graphs \cite{HL}, ($P_5,
\overline{P_3+O_2}$)-free graphs \cite{M}, ($P_5,
\overline{P_3+P_2}$)-free graphs \cite{ML}, and for ($P_5,
K_5-e$)-free graphs \cite{ML}.  In particular, the complexity
dichotomy of the VC problem is known for classes of graphs which are
defined by two connected five-vertex forbidden induced subgraphs
except for the following seven cases: (fork, bull)-free graphs, and
($P_5, H$)-free graphs, where $H\in \{\overline{K_3+O_2}, K_{2, 3},
\mbox{dart,~banner,~bull}, \overline{2P_2+P_1}\}$.  

 \begin{figure}
\centering
 \includegraphics{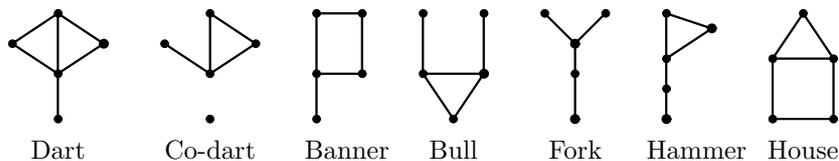}
\caption{Some special graphs.}\label{fig:sg}
\end{figure}

In the weighted version of the problem, we are given a graph and an
integer weight function $w$ on $V(G)$, and the \textsc{Minimum
Weighted Vertex Coloring} (WVC) problem is to find $k$ not necessarily
different stable sets $S_1, S_2, \ldots, S_k$ such that every vertex
$x$ belongs to at least $w(x)$ of these sets.  The smallest such $k$
is denoted by $\chi_w(G)$ and is called the \emph{weighted chromatic
number} of $G$; the stable sets $S_i$ are called a \emph{weighted
coloring} of $G$.  Note that in the context of the WVC problem, the
size of the input is considered to be $|V(G)|+|E(G)|+\sum_{x\in
V(G)}w(x)$.  Hence, algorithms for weighted graphs are polynomial on
the sum of weights but not necessarily on the number of vertices.
However, whenever we want to solve the (unweighted) coloring problem
on a graph $G$ and we reduce it to a weighted problem on an another
graph, it will always be the case that the size of the reduced
instance is not larger than $|V(G)|+|E(G)|$ (as can be easily checked,
since the reduction usually consists in replacing a subset $X$ of
vertices with one vertex of weight at most $|X|$); hence the final
complexity will be polynomial in $|V(G)|$.

 In this paper, using decomposition techniques, we establish structure
 theorems and derive the polynomial time solvability of the WVC
 problem for the following classes (see Figure~\ref{fig:sg}), which are four
 of the seven open cases mentioned above:
 \begin{itemize}
 \item ($P_5$, dart)-free graphs,
 \item ($P_5$, banner)-free graphs,
 \item $(P_5$, bull)-free graphs,
 \item (fork, bull)-free graphs.
 \end{itemize}

Rao \cite{Rao} showed that the VC problem is polynomial-time solvable
for graphs of bounded cliquewidth.  Note that the cliquewidth is
unbounded for each of the four classes above.  This is because each
class contains the class of co-triangle-free graphs, which has
unbounded cliquewidth since even the class of bipartite graphs has
unbounded cliquewidth (see \cite{MR}).

\section{Notation and preliminaries}

For notation and terminology which are not defined here, we follow
\cite{BLeS}.  All our graphs are finite, simple and undirected.  For
any integer $k$, we write $[k]$ to denote the set $\{1, 2, \ldots,
k\}$.  Let $K_n$, $P_n$, $C_n$ and $O_n$ denote respectively the
complete graph, the chordless path, the chordless cycle, and the
edgeless graph on $n$ vertices.  The graph $K_3$ is usually called a
\emph{triangle}.  A \emph{triad} in a graph $G$ is a subset of three
mutually non-adjacent vertices.

Given two vertex-disjoint graphs $G_{1}$ and $G_{2}$, the graph $G_1+
G_2$ is the graph with $V(G_1+ G_2)$ $=V(G_1)\cup V(G_2)$ and
$E(G_{1}+ G_{2})$ $=$ $E(G_1)\cup E(G_2)$.  For any positive integer
$k$, $kG$ denotes the union of $k$ graphs each isomorphic to $G$.  For
a graph $G$, the complement of $G$ is denoted by $\overline{G}$.  The
graph $\overline{P_5}$ is usually called the \emph{house}.  The
complement of a dart is called a \emph{co-dart}.  See
Figure~\ref{fig:sg} for some of the special graphs used in this paper.

Let $G$ be a graph.  For $S\subseteq V(G)$ we denote by $G[S]$ the
induced subgraph of $G$ with vertex-set $S$, and we simply write $[S]$
instead of $G[S]$.  If $H$ is an induced subgraph of $G$, then we
write $H \sset G$.  For a vertex $v \in V(G)$, the neighborhood $N(v)$
of $v$ is the set $\{u \in V(G) \mid uv \in E(G)\}$.  Given a subset
$S \subseteq V(G)$ and $v \in V(G) \sm S$, let $N_S(v)$ denote the set
$N(v) \cap S$.

For $v \in V(G)$ and $S\subseteq V(G)$, we say that a vertex $v$ is
\emph{complete} to $S$ if $v$ is adjacent to every vertex in $S$, and
that $v$ is \emph{anticomplete} to $S$ if $v$ has no neighbor in $S$.
For two sets $S,T\subseteq V(G)$ we say that $S$ is \emph{complete} to
$T$ if every vertex of $S$ is adjacent to every vertex of $T$, and we
say that $S$ is \emph{anticomplete} to $T$ if no vertex of $S$ is
adjacent to any vertex of $T$.

A {\it clique} in a graph $G$ is a subset of pairwise adjacent
vertices in $G$. The maximum size of a clique in $G$ is denoted by
$\omega(G)$.    A \emph{clique cover} of a graph $G$ is a partition
of $V(G)$ into cliques.  Hence of a coloring of a graph $G$ is a
clique cover of $\overline{G}$ and vice-versa.

A \emph{hole} in a graph is an induced cycle on at least five
vertices, and an \emph{anti-hole} is the complement of a hole. 
The length of a hole or anti-hole is the number of vertices in it.
A hole or anti-hole is \emph{odd} if its length is odd.  A graph $G$ is
\emph{perfect} if $\chi(H)=\omega(H)$ for every induced subgraph $H$
of $G$.  The Strong Perfect Graph Theorem (SPGT)
\cite{spgt} states that a graph $G$ is perfect if and only if it does
not contain an odd hole or an odd anti-hole.

A {\it clique separator} (or \emph{clique cutset}) in a connected
graph $G$ is a subset $Q$ of vertices in $G$ such that $Q$ is a clique
and such that the graph induced by $V(G) \sm Q$ is disconnected.  For
a given graph $G$, a \emph{$C$-block} is a maximal induced subgraph of
$G$ without proper clique separators.  For a class of graphs
$\cal{X}$, let $[\cal{X}]_C$ denotes the set of all graphs whose every
$C$-block belongs to $\cal{X}$.
\begin{theorem}[\cite{ML}]\label{thm:ML-2}
If the WVC problem can be solved in polynomial time for a hereditary
class $\cal{X}$, then it is so for $[\cal{X}]_C$.  \hfill {$\Box$}
\end{theorem}

A \emph{homogeneous set} in a graph $G$ is a set $S\subseteq V(G)$
such that every vertex in $V(G)\setminus S$ is either complete or
anticomplete to $S$.  A homogeneous set is \emph{proper} if it
contains at least two vertices and is different from $V(G)$.  A
\emph{module} is a homogeneous set $M$ such that every homogeneous set
$S$ satisfies either $S\subseteq M$ or $M\subseteq S$ or $S\cap
M=\emptyset$.  In particular $V(G)$ is a module and every one-vertex
set is a module.  The {\it trivial modules} in $G$ are $V(G)$,
$\emptyset$, and all one-elementary vertex sets.  A graph $G$ is {\it
prime} if it contains only trivial modules.  Note that prime graphs
with at least three vertices are connected.  It follows from their
definition that the modules form a ``nested'' family, so their
inclusion relation can be represented by a tree, and any graph $G$ has
at most $2|V(G)|-1$ modules.  The modules of a graph $G$ can be
produced by an algorithm of linear (i.e., $O(|V(G)|+|E(G)|)$) time
complexity \cite{CHPT,CH,MS}.  For a class of graphs $\cal{X}$, let
$[\cal{X}]_P$ denotes the set of all graphs whose every prime induced
subgraph belongs to $\cal{X}$.

\begin{theorem}[\cite{ML}, see also \cite{HL}]\label{thm:ML-1}
If the WVC problem can be solved in polynomial time for a hereditary
class $\cal{X}$, then it is so for $[\cal{X}]_P$.  \hfill {$\Box$}
\end{theorem}

\section{WVC for ($P_5$, dart)-free graphs}

Before considering ($P_5$, dart)-free graphs, we first look at ($P_5,
C_5$, dart)-free graphs and their complements.
 
\begin{theorem}\label{thm:hcd}
Let $G$ be any prime (house, co-dart)-free graph that contains an odd
hole of length at least~$7$.  Then $G$ is triangle-free.
\end{theorem}
\begin{proof}
Let $\ell$ be the length of any odd hole of length at least~$7$ in
$G$.  It follows that there exist $\ell$ non-empty and pairwise
disjoint subsets $A_1, \ldots, A_\ell$ of $V(G)$ such that, for each
$i$ modulo~$\ell$, the set $A_i$ is complete to $A_{i+1}$, and there
are no other edges between any two of these sets.  Let
$A=A_1\cup\cdots\cup A_\ell$.  We choose these sets so that $A$ is
inclusionwise maximal.  Let $B$ be the set of vertices of
$V(G)\setminus A$ that are complete to $A$.  We claim that:
\begin{equation}\label{hcd1}
\mbox{Each $A_i$ is a stable set.}
\end{equation}
Proof: Suppose that $u,v$ are two adjacent vertices in $A_i$.  Pick an
arbitrary vertex $a_j\in A_j$ for each $j\in\{i+1, i+2, i+4\}$.  Then
$\{u,v,a_{i+1}, a_{i+2},a_{i+4}\}$ induces a co-dart.  Thus
(\ref{hcd1}) holds.

\begin{equation}\label{hcd2}
\mbox{For any vertex $v\in V(G)\setminus (A\cup B)$, the set
$N_A(v)$ is a stable set.}
\end{equation}
Proof: Suppose the contrary.  By (\ref{hcd1}), and since $v\notin B$,
we can pick a vertex $a_{i}\in A_{i}$ for each $i$ such that $v$ is
adjacent to $a_1$ and $a_2$ and not to $a_3$.  Then $v$ is adjacent to
$a_i$ for all $i\in\{5,...,\ell-1\}$, for otherwise
$\{a_1,a_2,a_3,v,a_i\}$ induces a co-dart; but then
$\{a_1,v,a_5,a_6,a_3\}$ induces a co-dart.  Thus (\ref{hcd2}) holds.

\begin{equation}\label{hcd3}
\mbox{$B=\emptyset$.}
\end{equation}
Proof: Suppose that $B\neq\emptyset$.  Consider any $b\in B$ and any
$v\in V(G)\setminus (A\cup B)$.  For each $i$ pick a vertex $a_i\in
A_i$.  By (\ref{hcd2}) and since $\ell$ is odd, there is an integer
$j$ such that $v$ has no neighbor in $\{a_j,a_{j+1}\}$, say $j=1$;
moreover $v$ has a non-neighbor $a$ in $\{a_4,a_5\}$.  Then $b$ is
adjacent to $v$, for otherwise $\{a,b,a_1,a_2,v\}$ induces a co-dart.
This means that $B$ is complete to $V(G)\setminus (A\cup B)$.  Since
$B$ is also complete to $A$, we deduce that $V(G)\setminus B$ is a
homogeneous set, which contradicts that $G$ is prime.  Thus
(\ref{hcd3}) holds.

\medskip

To finish the proof of the theorem, let us assume on the contrary that
$G$ contains a triangle $T=\{u,v,w\}$.  By~(\ref{hcd1}), the graph
$G[A]$ is triangle-free.  Moreover, by~(\ref{hcd2}) and (\ref{hcd3}),
no triangle of $G$ has two vertices in $A$.  So $T$ contains at most
one vertex from $A$.  Note that $G$ is connected, for otherwise the
vertex-set of the component that contains $A$ would be a proper
homogeneous set.  So there is a shortest path $P$ from $A$ to $T$.
Let $P=p_1$-$\cdots$-$p_k$, with $p_1\in A$, $p_2, \ldots, p_k\in
V(G)\setminus A$, $p_k=u$, $k\ge 1$, and $v,w\notin A$.  We choose $T$
so as to minimize $k$.  We can pick vertices $a_i\in A_i$ for each
$i\in\{1,...,\ell\}$ so that $p_1=a_1$.  If $k\ge 3$, then
$\{p_1,p_{k-1},u,v,w\}$ induces a co-dart.  So $k\le 2$.  Suppose that
$k=1$.  So $u=p_1=a_1$.  By (\ref{hcd2}) and (\ref{hcd3}), $v$ and $w$
have no neighbor in $\{a_2,a_\ell\}$.  If $a_3$ has no neighbor in
$\{v,w\}$, then $\{a_3,a_\ell,a_1,v,w\}$ induces a co-dart.  If $a_3$
has only one neighbor in $\{v,w\}$, then $\{a_3, a_2, a_1, v, w\}$
induces a house.  So $a_3$ is adjacent to both $v,w$, and
$\{a_3,v,w\}$ is a triangle.  Then we can repeat this argument with
$a_3, a_5, ...$ instead of $a_1$, which leads to a contradiction since
$\ell$ is odd.  Therefore $k=2$.  So $u=p_2$.  Since $\ell$ is odd,
and by (\ref{hcd2}) and (\ref{hcd3}), we may assume that $u$ is
adjacent to $a_1$ and not to $a_2$ and $a_3$.  Then $a_3$ has a
neighbor in $\{v,w\}$, say $v$, for otherwise $\{a_3,a_1,u,v,w\}$
induces a co-dart, and then $a_3$ is not adjacent to $w$, for
otherwise $\{a_3,v,w\}$ is a triangle that contradicts the minimality
of $k$.  Then $w$ is not adjacent to $a_4$, by (\ref{hcd2}) and
(\ref{hcd3}), and $a_4$ has no neighbor in $\{u,v\}$, for otherwise
either $\{a_4,u,v\}$ is a triangle (contradicting the minimality of
$k$) or $\{a_3,a_4,u,v,w\}$ induces a house.  But then
$\{a_4,a_1,u,v,w\}$ induces a co-dart, a contradiction.  This
completes the proof of the theorem.
\end{proof}

\begin{theorem}\label{thm:P5C5D-free-compl}
The \textsc{Weighted Vertex Coloring} problem can be solved in
polynomial time in the class of ($P_5, C_5$, dart)-free graphs.
 \end{theorem}
\begin{proof} 
Let $G$ be a ($P_5, C_5$, dart)-free graph.  First suppose that $G$ is
prime.  If $G$ has no odd anti-hole of length at least 7, then by SPGT
\cite{spgt}, $G$ is perfect.  Otherwise, $G$ is $O_3$-free, by
Theorem~\ref{thm:hcd}.  Since the class of perfect graphs can be
recognized in polynomial time \cite{perf-rec}, and since the WVC
problem can be solved in polynomial time for perfect graphs
\cite{GLS}, and for $O_3$-free graphs \cite{ML}, WVC can be solved in
polynomial time for $G$.  By Theorem~\ref{thm:ML-1}, the same holds
when $G$ is not prime.
 \end{proof}

 \begin{figure}[htb]
\centering
\includegraphics{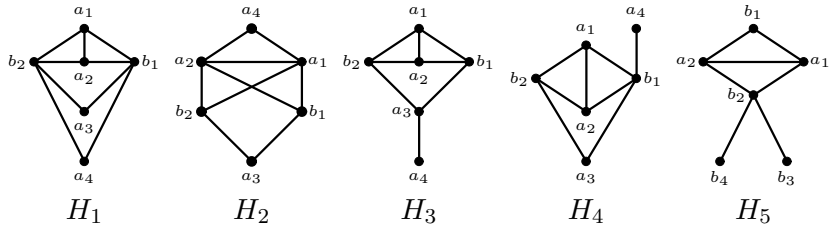}
\caption{Special graphs used in Lemmas~\ref{lem:H1H2-free}
and~\ref{lem:H3H4-free}.}\label{fig:H12345}
\end{figure}

\begin{lemma}\label{lem:H1H2-free}
Let $G$ be a prime dart-free graph.  Then $G$ is $(H_1, H_2)$-free.
\end{lemma}
\begin{proof}
Suppose to the contrary that $G$ contains an induced subgraph which is
isomorphic to $H_i$, for $i \in \{1, 2\}$ (as shown in
Figure~\ref{fig:H12345}).  Since $G$ is prime, $\{a_1, a_2\}$ is not a
module, so there exists a vertex $x \in V \setminus V(H_i)$ such that
(up to symmetry) $xa_1 \in E$ and $xa_2 \notin E$.  Then since $\{x,
a_1, b_1, b_2, a_2\}$ does not induce a dart, $x$ has a neighbor in
$\{b_1, b_2\}$.  By symmetry, we may assume that $xb_1 \in E$.  Then
since $\{x, a_1, a_2, b_1, a_3\}$ does not induce a dart, $xa_3 \in
E$.

\no Suppose that $i = 1$.  Then since $\{x, a_1, a_2, b_1, a_4\}$ does
not induce a dart, $xa_4 \in E$.  But, then $\{a_4, x, a_3, b_1,
a_2\}$ induces a dart.  So, $G$ is $H_1$-free.

\no Suppose that $i = 2$.  Then since $\{b_2, a_2, a_4, a_1, x\}$ and
$\{b_1, x, a_4, a_1, b_2\}$ do not induce a dart, we have $xb_2 \in
E$.  But, then $\{b_1, x, b_2, a_1, a_4\}$ or $\{b_1, a_3, b_2, x,
a_4\}$ induces a dart.  So, $G$ is $H_2$-free.

\no  This shows Lemma~\ref{lem:H1H2-free}. 
\end{proof}

\begin{lemma}\label{lem:H3H4-free}
Let $G$ be a prime ($P_5$, dart)-free graph.  Then $G$ is $(H_3, H_4,
H_5)$-free.
\end{lemma}
\begin{proof} 
Suppose to the contrary that $G$ contains an induced subgraph which is
isomorphic to $H_i$, for $i \in \{3, 4, 5\}$ (as shown in
Figure~\ref{fig:H12345}).  Since $G$ is prime, $\{a_1, a_2\}$ is not a
module, so there exists a vertex $x \in V \setminus V(H_i)$ such that
(up to symmetry) $xa_1 \in E$ and $xa_2 \notin E$.  Then since $\{x,
a_1, b_1, b_2, a_2\}$ does not induce a dart, $x$ has a neighbor in
$\{b_1, b_2\}$.

Suppose that $i =3$.  By symmetry, we may assume that $xb_1 \in
E$.  Since $\{a_1, x, b_1, a_2, a_3\}$ does not induce a dart, $xa_3
\in E$.  Since $\{a_2, a_1, x, a_3, a_4\}$ does not induce a $P_5$,
$xa_4 \in E$.  Since $\{b_1, x, a_4, a_3, b_2\}$ does not induce a
dart, $xb_2 \in E$.  But, then $\{b_1, a_1, b_2, x, a_4\}$ induces a
dart.  So, $G$ is $H_3$-free.

Suppose that $i =4$.  First suppose that $xb_1\in E$.  Since $\{x,
a_1, a_2, b_1, a_3\}$ and $\{x, a_1, a_2, b_1, a_4\}$ do not induce a
dart, we have $xa_3 \in E$ and $xa_4 \in E$.  But, then $\{x, a_3,
b_1, a_4, a_2\}$ induces a dart.  So, $xb_1 \notin E$, and hence $xb_2
\in E$.  Since $\{x, a_1, a_2, b_2, a_3\}$ does not induce a dart,
$xa_3 \in E$, and since $\{x, b_2, a_2, b_1, a_4\}$ does not induce a
$P_5$, $xa_4 \in E$.  But, then $\{a_1, b_2, a_3, x, a_4\}$ induces a
dart.  So, $G$ is $H_4$-free.

Finally, suppose that $i =5$.  First suppose that $xb_2 \in E$.  Since
$\{x, a_1, a_2, b_2, b_3\}$ and $\{x, a_1, a_2, b_2, b_4\}$ do not
induce a dart, we have $xb_3 \in E$ and $xb_4 \in E$.  But, then $\{x,
b_2, b_3, b_4, a_2\}$ induces a dart.  So, $xb_2 \notin E$, and hence
$xb_1 \in E$.  Since $\{x, b_1, a_2, b_2, b_3\}$ and $\{x, b_1, a_2,
b_2, b_4\}$ do not induce a $P_5$, we have $xb_3 \in E$ and $xb_4 \in
E$.  Again, since $G$ is prime, $\{b_3, b_4\}$ is not a module in $G$,
so there exists a vertex $y \in V \setminus V(H_5)$ such that (up to
symmetry) $yb_3 \in E$ and $yb_4 \notin E$.  Note that $y \neq x$.
Suppose that $yb_2 \notin E$.  If $ya_1, ya_2 \in E$, then $\{y, a_1,
a_2, b_2, b_3, b_4\}$ induces an $H_4$, and if $ya_1 \notin E$, then
since $\{y, b_3, b_2, a_1, b_1\}$ does not induce a $P_5$, $yb_1 \in
E$.  But, then $\{b_4, b_2, a_1, b_1, y\}$ induces a $P_5$ (and
similar proof holds when $ya_2 \notin E$).  So, $yb_2 \in E$.  Then
$ya_1, ya_2 \notin E$ (otherwise, $\{y, a_1, b_2, b_3, b_4\}$ or $\{y,
a_2, b_2, b_3, b_4\}$ induces a dart).  Then since $\{x, y, a_1, b_2,
b_3, b_4\}$ does not induce an $H_1$, we have $xy \notin E$.  Then
since $\{y, b_2, a_2, b_1, x\}$ does not induce a $P_5$, $yb_1 \in E$.
But, then $\{x, a_1, a_2, b_1, y\}$ induces a dart.  So, $G$ is
$H_5$-free.  

This shows Lemma~\ref{lem:H3H4-free}.
\end{proof}

By Lemmas~\ref{lem:H1H2-free} and~\ref{lem:H3H4-free} prime ($P_5$,
dart)-free graphs are $(H_1, H_2, H_3, H_4, H_5)$-free; see
Figure~\ref{fig:H12345}.

\begin{theorem} \label{thm:P5D-free-str}
Let $G$ be a prime ($P_5$, dart)-free graph that contains an induced
$C_5$.  Then either $|V(G)| \leq 18$ or $G$ is $O_3$-free.
\end{theorem}

\begin{proof} 
Consider an induced $C_5$ in $G$, with vertex-set $C := \{v_1, v_2,
v_3, v_4, v_5\}$ and edge-set $\{v_iv_{i+1}|i \in [4]\} \cup
\{v_5v_1\}$.  We associate the following notation with $G$, taking the
indices modulo $5$:
\begin{eqnarray*}
W_i &:=& \{x \in V \setminus C \mid N_C(x) = \{v_{i-1}, v_{i+1}\}\},\\
X_i &:=& \{x \in V \setminus C \mid N_C(x) = \{v_{i-1}, v_i, v_{i+1}\}\},\\
Y_i &:=& \{x \in V \setminus C \mid N_C(x) = \{v_{i}, v_{i+1}, v_{i-2}\}\},\\
Z_i &:=& \{x \in V \setminus C \mid N_C(x) = V(C) \sm \{v_{i}\}\},\\
T_{\ } &:=& \{x \in V \setminus C \mid N_C(x) = V(C)\},\\
R_{\ } &:=& \{x \in V \setminus C \mid N_C(x) = \emptyset\}.
\end{eqnarray*}

\no Let $W := \cup_{i = 1}^5 W_i$, $X := \cup_{i = 1}^5 X_i$, $Y :=
\cup_{i = 1}^5 Y_i$, and $Z := \cup_{i = 1}^5 Z_i$.

\medskip

\no For every $i\in [5]$ (mod $5$) the following properties hold true.
\begin{enumerate}
\item[(1)]If $x \in V\sm C$ and if $N_C(x) \neq \es$, then $x \in W
\cup X \cup Y \cup Z \cup T$, as $G$ is $P_5$-free.  \item[(2)] $W_i$
is an independent set.

\item[(3)]
\begin{enumerate}
\item $X_i$ and $Z_{i}$ are cliques.
\item $X_i$ is complete to $X_{i+1}$.
\item $X_i$ is complete to $Z_{i-2} \cup Z_{i+2}$.
\item If $X \neq \es$, then $W\cup Y = \es$.
\end{enumerate}

\item[(4)]
\begin{enumerate}
\item $|Y_i| \leq 1$.
\item If $Y_i \neq \es$, then $Y_{i+1} = \es = Y_{i-1}$.
\item $|Y| \leq 2$.
\item $Y_i$ is anticomplete to $Y_{i+2}$.
\end{enumerate}

\item[(5)]
\begin{enumerate}
\item If $Z \neq \es$, then $W=\es$.

\item If $Z_i \neq \es$, for some $i$,  then $Y \sm Y_{i+2} = \es$.

\item $Z_i$ is anticomplete to $Y_{i+2}$.

\end{enumerate}

\item[(6)]
\begin{enumerate}
\item If $T \neq \es$, then $W\cup Y = \es$.
\item $[T]$ is $O_3$-free.
\end{enumerate}

\item[(7)]
\begin{enumerate}
\item If $x \in W \cup X \cup Z \cup T$, then $N(x) \cap R = \es$.
So, if $R \neq \es$, then since $G$ is connected, $Y \neq \es$.  \item
Let $x \in Y_i$ and $y \in Y_{i+2}$.  Then $x$ and $y$ have the same
neighbors in $R$.  \item $|R| \leq 1$.  \item If $Z_i \neq \es$, for
some $i$, then $R = \es$.
\end{enumerate}

\end{enumerate}

\no{\it Proof of properties $(1)$--$(7)$.} Some of the above properties
can be verified routinely and in that case we omit their proof;
however, we do give a proof for those properties which are not
trivial.

\no (2): If there are adjacent vertices $x$ and $y$ in $W_i$, then
$\{x, y, v_i, v_{i+1}, v_{i-1}, v_{i+2}\}$ induces an $H_4$, which
contradicts Lemma~\ref{lem:H3H4-free}.

\no (3:a): If there are non-adjacent vertices $x$ and $y$ in $X_i$ or
in $Z_i$, then $\{x, y, v_i, v_{i+1}, v_{i+2}\}$ induces a dart.

\no (3:b): If there are non-adjacent vertices $x \in X_i$ and $y\in
X_{i+1}$, then $\{x, v_{i-1}, v_{i-2}, v_{i+2}, y\}$ induces a $P_5$.

\no (3:c): If there are non-adjacent vertices $x \in X_i$ and $y\in
Z_{i+2} \cup Z_{i-2}$, say, by symmetry, $y \in Z_{i+2}$, then $\{x,
y, v_{i}, v_{i+1}, v_{i+2}\}$ induces a dart.

\no (3:d): Suppose not.  Let $x \in X_1 \subseteq X$ and let $w \in W
\cup Y$.  Then up to symmetry we have the following cases.  If $w \in
W_1$, then $\{x, w, v_1, v_2, v_3, v_5\}$ induces a dart or an $H_4$.
If $w \in W_2$, then $\{x, w, v_1, v_2, v_3, v_4\}$ induces an $H_3$
(if $xw \in E$), or $\{x, v_1, v_2, v_5, w\}$ induces a dart (if $xw
\notin E$).  If $w \in W_3$, then $\{x, w, v_1, v_2, v_3\}$ induces a
dart (if $xw \in E$), or $\{x, v_1, v_2, v_3, v_5, w\}$ induces a
$H_5$ (if $xw \notin E$).  If $w \in Y_1$, then $\{x, w, v_1, v_3,
v_4, v_5\}$ induces an $H_3$ (if $xw \in E$), or $\{x, w, v_1, v_2,
v_3\}$ induces a dart (if $xw \notin E$).  If $w \in Y_2$, then $\{x,
w, v_1, v_2, v_4 , v_5\}$ induces a dart or an $H_4$.  If $w \in Y_3$,
then $\{x, w, v_1, v_2,$ $ v_4, v_5\}$ induces a dart or $H_2$.  So,
(3:d) holds.

\no (4:a): Suppose not.  Then there are two vertices $x$ and $y$ in
$Y_i$.  Now, if $xy \in E$, then $\{x, y, v_{i+1}, v_{i+2},$ $
v_{i-2}, v_{i-1}\}$ induces an $H_4$, and if $xy \notin E$, then $\{x,
y, v_i, v_{i+1}, v_{i+2}\}$ induces a dart.  So, (4:a) holds.

\no (4:b): Suppose not.  Then there are two vertices $x \in Y_i$ and
$y \in Y_{i+1} \cup Y_{i-1}$, say, up to symmetry, $y \in Y_{i+1}$.
Now, if $xy \in E$, then $\{x, v_i, v_{i+1}, y, v_{i-2}\}$ induces a
dart, and if $xy \notin E$, then $\{v_i, x, v_{i-2}, v_{i+2}, y\}$
induces a $P_5$.  So, (4:b) holds.

\no (4:c): The proof follows from (4:a) and (4:b).

\no (4:d): If there are adjacent vertices $x \in Y_i$ and $y \in
Y_{i+2}$, then $\{x, y, v_{i+2}, v_{i-2}, v_{i-1}\}$ induces a dart.

\no (5:a): Suppose not.  Let $z \in Z_1$ and let $w \in W$.  Up to
symmetry, we have the following cases.  If $w \in W_1$, then $\{z, w,
v_{i+2}, v_{i+1}, v_i\}$ induces a dart (if $zw \in E$), or
$\{v_{i+2}, v_{i-2}, z, v_{i-1}, v_i, w\}$ induces an $H_5$ (if $zw
\notin E$).  If $w \in W_2$, then $\{z, w, v_{i}, v_{i+1}, v_{i+2},
v_{i-2}\}$ induces an $H_2$ (if $zw \in E$), or $\{z, v_{i+1}, v_{i+2},
v_{i-2}, w\}$ induces a dart (if $zw \notin E$).  If $w\in W_3$, then
$\{z, w, v_{i+1}, v_{i+2}, v_i\}$ induces a dart (if $zw \in E$), or
$\{z, v_{i+2}, v_{i-2}, v_{i-1}, w\}$ induces a dart (if $zw \notin
E$).  So, (5:a) holds.

\no (5:b): Suppose not.  Let $z \in Z_1$ and let $y \in Y \sm Y_3$.
Then up to symmetry we have the following cases.  If $y \in Y_1$, then
$\{z, y, v_i, v_{i+2}, v_{i-2}, v_{i-1}\}$ induces an $H_2$ (if $zy \in
E$), or $\{z, v_{i+2}, v_{i-2}, v_{i-1}, y\}$ induces a dart (if $zy
\notin E$).  If $y \in Y_2$, then $\{z, y, v_{i-1}, v_{i-2}, v_{i}\}$
induces a dart (if $zy \in E$), or $\{z, y, v_{i+1}, v_{i+2}, v_{i}\}$
induces a dart (if $zy \notin E$).  Thus, (5:b) holds.

\no (5:c): If there are adjacent vertices $z \in Z_i$ and $y \in
Y_{i+2}$, then $\{z, y, v_{i-2}, v_{i-1}, v_{i+1}\}$ induces a dart.

\no (6:a): Suppose not.  Let $t \in T$.  Now, if $w \in W_1$, then
$\{t, w, v_1, v_2, v_3\}$ or $\{t, w, v_1, v_2, v_4\}$ induces a dart,
and if $y \in Y_1$, then $\{t, y, v_3, v_4, v_5\}$ or $\{t, y, v_2,
v_3, v_5\}$ induces a dart.  Since the other cases are symmetric,
these contradictions show that (6:a) holds.

\no (6:b): Suppose to the contrary that $[T]$ contains a triad, say $S
: = \{a, b, c\}$.  Consider the co-connected component of $[T]$
containing $S$.  Since $G$ is prime, $\{a, b\}$ is not a module in
$G$, so there exists a vertex $x$ in $X \cup Z$ (by (6:a)) that is
adjacent to $a$ and not to $b$.  Since $x \in X\cup Z$, there exist
$i, j \in [5]$ such that $v_iv_j \in E$, $xv_i \in E$ and $xv_j \notin
E$.  Then since $\{x, a, b, c, v_i\}$ does not induce a dart, $xc
\notin E$.  But, then $\{x, b, c, v_i, v_j\}$ induces a dart.  So,
(6:b) holds.

\no (7:a): Suppose not.  Let $x \in W \cup X \cup Z \cup T$ be such
that $N(x) \cap R \neq \es$.  Let $y \in N(x) \cap R$.  Now, if $x \in
X \cup Z \cup T$, then there exists $i \in [5]$ such that $\{v_i,
v_{i+1}, v_{i+2}\} \subseteq N(x)$.  But, then $\{v_i, v_{i+1},
v_{i+2}, x, y\}$ induces a dart.  So, $x \in W$, say $x \in W_1$.
Then $\{y, x, v_2, v_3, v_4\}$ induces a $P_5$.

\no (7:b): Suppose not.  Then up to symmetry, let $r\in R$ be such
that $rx \in E$ and $ry \notin E$.  But, then since $xy \notin E$ (by
(4:d)), $\{r, x, v_{i+1}, v_{i+2}, y\}$ induces a $P_5$.

\no (7:c): Since $G$ is prime, the proof follows from (7:a) and (7:b).

\no (7:d): Suppose not.  Let $z \in Z_i$ and $r\in R$.  Then by
(7:a), $zr \notin E$ and $Y \neq \es$, and by (5:b) we obtain $Y \sm
Y_{i+2} = \es$.  Thus, $Y_{i+2} \neq \es$.  Let $y \in Y_{i+2}$ be
such that $yr \in E$.  Then, by (5:c), $zy \notin E$.  But, then
$\{v_{i-1}, z, v_{i+2}, y, r\}$ induces a $P_5$.

\medskip

Moreover, the following holds.
\begin{enumerate}
\item[(8)] For every $i \in [5]$, we have $|W_i| \leq 2$. So, $|W| \leq 10$.
\end{enumerate}
Proof: We prove for $i =1$.  Suppose to the contrary that $|W_1| \geq
3$.  Then by $(2)$, there exist three mutually non-adjacent vertices
in $W_1$, say $w_1$, $w_2$ and $w_3$.  Let $G'$ be the graph induced
by $V(C) \cup \{w_1, w_2, w_3\}$, and let $W' := \{v_1, w_1, w_2,
w_3\}$.  Since $G$ is prime, $V \sm V(G') \neq \es$.  Then we have the
following claim.

\begin{claim}\label{clm:pr1}
 Let $x \in V\sm V(G')$.  Suppose that $x$ has a neighbor and a
 non-neighbor in $W'$. Then: (i) $x$ has exactly one neighbor in
 $\{v_2, v_5\}$, and (ii) $|N(x) \cap W'| = 1$.
 \end{claim}

\no{\it Proof of Claim~$\ref{clm:pr1}$}.  We may assume, up to
symmetry, that $xw_1 \in E$ and $xw_2 \notin E$.

($i$):  If $x$ has no neighbor in $\{v_2, v_5\}$,
then since $\{w_2, v_2, w_1, x, v_4\}$ does not induce a $P_5$, $xv_4
\notin E$.  Similarly, since $\{w_2, v_5, w_1, x, v_3\}$ does not
induce a $P_5$, $xv_3 \notin E$.  Then $\{x, w_1, v_2, v_3,
v_4\}$ induces a $P_5$ in $G$, a contradiction.  If $x$ is adjacent to
both $v_2$ and $v_5$, then since $\{x, w_1, v_2, v_3, w_2\}$ does not
induce a dart, $xv_3 \notin E$.  But then $\{x, v_2, v_3, v_5, w_1,
w_2\}$ induces an $H_4$, which contradicts Lemma~\ref{lem:H3H4-free}.
So $(i)$ holds.

($ii$): By our assumption, $w_1 \in N(x) \cap W'$.  We show that
$N(x) \cap W' = \{w_1\}$.  Suppose not.  Up to symmetry, we may assume
that $xw_3 \in E$.  By ($i$), $x$ has exactly one neighbor in $\{v_2,
v_5\}$, say, by symmetry, $xv_2 \in E$.  Now, $\{x, v_2, w_1, w_3,
w_2\}$ induces a dart.  So, $(ii)$ holds.  This shows
Claim~\ref{clm:pr1}.

\smallskip

Since $G$ is prime, $W'$ is not a module.  So, by
Claim~\ref{clm:pr1}(ii) and by symmetry, there are vertices $x$, $y$,
$z$ in $V\sm V(G')$ such that $N(x) \cap W' = \{v_1\}$, $N(y) \cap W'
= \{w_1\}$, and $N(z) \cap W' = \{w_2\}$.  Then by
Claim~\ref{clm:pr1}(i), two of the vertices in $\{x, y, z\}$ have the
same neighbor in $\{v_2, v_5\}$.  Up to symmetry, let $x$ and $y$ have
the same neighbor $v_2$.  Again by Claim~\ref{clm:pr1}(i), $xv_5
\notin E$ and $yv_5 \notin E$.  Then since $\{x, y, v_1, v_2, w_3\}$
does not induce a dart, $xy \notin E$.  Also, since the subgraph
induced by $\{x, y, v_1, v_2, v_3, w_1, w_3\}$ does not contain a
dart, $xv_3 \notin E$ and $yv_3 \notin E$.  Then since the subgraph
induced by $\{x, y, v_1, v_2, v_3, v_4, w_1\}$ does not contain a
$P_5$, we have $xv_4 \in E$ and $yv_4 \in E$.  But then $\{v_1, x,
v_4, y, w_1\}$ induces a $P_5$.  Thus, (8) holds.

\medskip

Now, we claim that:
\begin{enumerate}
\item[(9)] $[C\cup X \cup Z]$ is $O_3$-free.
\end{enumerate}
Proof: Suppose to the contrary that $[C\cup X \cup Z]$ contains a
triad, say $S : = \{a, b, c\}$.  By the definitions of $X$ and
$Z$, the set $S$ has at most one vertex from $C$.  If $v_1 (:= a) \in S$
(say), then $b$ and $c$ belong to $X_3 \cup X_4 \cup Z_1$, which is
impossible, by (3).  So, suppose that none of the vertices
from $V(C)$ belongs to $S$.  Then by using $(3)$, we have the
following cases (the other cases are either similar or symmetric): 
(i)~$a \in X_1$, $b \in X_3\cup Z_1$ and $c \in Z_2$. Then $\{a, v_1,
c, v_3, b\}$ induces a $P_5$. (ii)  $a \in X_1$, $b \in Z_2$ and $c
\in Z_5$. Then $\{v_3, v_4, b, c, v_1, a\}$ induces an $H_3$. (iii)
$a \in Z_1$, $b \in Z_2$ and $c \in Z_3 \cup Z_4$. Then $\{a, b, c,
v_1, v_5\}$ or $\{a, b, c, v_4, v_5\}$ induces a dart.  Thus
(9) holds.

\medskip

Now, if $X\cup Z \cup T = \es$, then by properties (4:c) and (7:d) and
by (8), it follows that $|V| = |C \cup W \cup Y \cup R| \leq 10+5+2+1
= 18$.  Therefore we may assume that $X\cup Z \cup T \neq \es$.

Suppose that $T \neq \es$.  Then we show that $G$ is $O_3$-free.
Suppose to the contrary that $G$ contains a triad $S:= \{a, b,
c\}$.  Since $T \neq \es$, we have $W = Y= \es$ by $(6)$.  So, $R =
\es$ by (7:a).  Also, since $[C\cup X \cup Z]$ is $O_3$-free, by
(9), at least one vertex from $T$ is in $S$, and hence none of the
vertices from $C$ belong to $S$.  Since $[T]$ is $O_3$-free, by the
above properties, we have the following cases (the other cases are
symmetric): \\
(i) $a \in X_1$, $b \in X_3$ and $c \in T$.  Then $\{a, v_1, c, v_3,
b\}$ induces a $P_5$.\\
(ii) $a \in X_1$, $b \in T$ and $c \in T$.  Then $\{a, b, c, v_2,
v_3\}$ induces a dart. \\ 
(iii)  $a \in Z_1$, $b \in Z_2\cup Z_3 \cup T$ and $c \in T$. Then
$\{a, b, c, v_1, v_5\}$ induces a dart.  \\
(iv) $a\in X_1$, $b\in Z_1\cup Z_2$, and $c\in T$. Then 
$\{a,b,c,v_4,v_5\}$ induces a dart. \\
These contradictions show that $G$ is $O_3$-free.  Therefore we may
assume that $T = \es$.  If $X \neq \es$, then since $W= Y = \es$, by
(3:d), and $R = \es$, by (7:a), it follows from (9) that $G$ is
$O_3$-free.  Therefore we may assume that $T \cup X= \es$.  Thus, $Z
\neq \es$, say $Z_1\neq\emptyset$.  Then by (5:a) and (5:b), $W= \es$
and $Y\sm Y_3 = \es$.  Also, by (7:d), $R = \es$.  If $Y_3 = \es$,
then by (9), $G$ is $O_3$-free.  So, suppose that $Y_3 \neq \es$.
Again, by (5:a), $Z\sm Z_1 = \es$.  Then since $Z_1$ is a clique, and
any vertex in $Y_3$ is adjacent to $v_1$, it follows that $G$ is
$O_3$-free.  This completes the proof of
Theorem~\ref{thm:P5D-free-str}.
\end{proof}

Finally we can prove the main result of this section.
\begin{theorem}\label{thm:P5D-free-compl}
The \textsc{Weighted Vertex Coloring} problem can be solved in
polynomial time in the class of ($P_5$, dart)-free graphs.
\end{theorem}
\begin{proof}
Since the class of perfect graphs can be recognized in polynomial time
\cite{perf-rec}, and the WVC problem can be solved in polynomial time
for perfect graphs \cite{GLS}, for $O_3$-free graphs \cite{ML}, and
for the graphs having at most $c$ vertices (for any fixed $c$), the
theorem follows from Theorems~\ref{thm:P5C5D-free-compl}
and~\ref{thm:P5D-free-str}.
\end{proof}


\section{WVC for ($P_5$, banner)-free graphs}

The main result of this section is the following.
\begin{theorem}\label{thm:pbamain}
The \textsc{Weighted Vertex Coloring} problem can be solved in
polynomial time in the class of ($P_5$, banner)-free graphs.
\end{theorem}

We first establish a structure theorem for the complement graph of a
($P_5$, banner)-free graph.  The complement of a banner is called a
\emph{hammer}.  See Figure~\ref{fig:sg}.

\begin{theorem}\label{thm:pba2}
Let $G$ be any prime (hammer, house)-free graph.  Then $G$ is either
perfect or triangle-free.
\end{theorem}
\begin{proof}
Let $G$ be a prime (house, hammer)-free graph, and suppose that $G$ is
not perfect.  By the Strong Perfect Graph Theorem $G$ contains an odd
hole or an odd antihole of length at least~$5$.  However, every
antihole of length at least $6$ contains a house.  So $G$ contains a
hole, of length $\ell\ge 5$.  It follows that there exist $\ell$
non-empty and pairwise disjoint subsets $A_1, \ldots, A_\ell$ of
$V(G)$ such that, for each $i$ modulo~$\ell$, the set $A_i$ is
complete to $A_{i+1}$, and there are no other edges between any two of
these sets.  Let $A=A_1\cup\cdots\cup A_\ell$.  We choose these sets
so that $A$ is inclusionwise maximal.  Let $B$ be the set of vertices
of $V(G)\setminus A$ that are complete to $A$.  We first claim that:
\begin{equation}\label{hhais}
\mbox{Each $A_i$ is a stable set.}
\end{equation}
Proof: Suppose that $u$ and $v$ are two adjacent vertices in $A_i$.
Pick an arbitrary vertex $a_j\in A_j$ for each $j\in\{i+1, i+2,
i+3\}$.  Then $\{u,v,a_{i+1}, a_{i+2},a_{i+3}\}$ induces a hammer.
Thus (\ref{hhais}) holds.

\medskip

Now we claim that:
\begin{equation}\label{hhn}
\mbox{For any vertex $v\in V(G)\setminus (A\cup B)$, the set
$N_A(v)$ is a stable set.}
\end{equation}
Proof: Suppose the contrary.  By (\ref{hhais}) there is an integer $i$
such that $v$ has a neighbor $a_i\in A_i$ and a neighbor $a_{i+1}\in
A_{i+1}$.  Consider arbitrary vertices $a_{i+2}\in A_{i+2}$ and
$a_{i+3}\in A_{i+3}$.  Then $v$ is adjacent to $a_{i+2}$, for
otherwise $\{v, a_i, a_{i+1}, a_{i+2}, a_{i+3}\}$ induces a hammer or
a house (depending on the adjacency between $v$ and $a_{i+3}$).  Now
we can repeat this argument with $i+1$ and so on, which implies that
$v\in B$, a contradiction.  Thus (\ref{hhn}) holds.

\medskip

Now we claim that:
\begin{equation}\label{hhbem}
\mbox{$B=\emptyset$.}
\end{equation}
Proof: Suppose that $B\neq\emptyset$.  Let $H$ be the component of
$G\setminus B$ that contains~$A$.  By the hypothesis, $V(H)$ is not a
proper homogeneous set, which implies that there exist non-adjacent
vertices $b\in B$ and $x\in V(H)$.  By the definition of $H$ there is
a shortest path $p_1$-$\cdots$-$p_k$ in $H$ with $p_1\in A$ and
$p_k=x$, and we choose the pair $b,x$ so as to minimize $k$.  We have
$k\ge 2$ since $x\notin A$.  We can pick vertices $a_i\in A_i$ for
each $i\in\{1,\ldots,\ell\}$ so that $p_2$ has a neighbor in the set
$\{a_1,...,a_\ell\}$.  Since $\ell$ is odd, and by (\ref{hhn}), $p_2$
has two consecutive non-neighbors in that set, so, up to relabeling,
we may assume that $p_2$ is adjacent to $a_1$ and not adjacent to
$a_2$ and $a_3$.  Then $b$ is adjacent to $p_2$, for otherwise $\{p_2,
a_1, b, a_3, a_4\}$ induces a hammer or a house (depending on the
adjacency between $p_2$ and $a_4$).  Hence there is an integer $j\le
k$ such that $b$ is adjacent to $p_j$ and not to $p_{j+1}$.  But then
$\{p_{j+1}, p_j, b, a_2, a_3\}$ induces a hammer.  Thus (\ref{hhbem})
holds.

\medskip

To finish the proof of the theorem, suppose on the contrary that $G$
contains a triangle $T=\{u,v,w\}$.  By~(\ref{hhais}) the graph $G[A]$
is triangle-free.  Moreover, by~(\ref{hhn}), no triangle of $G$ has
two vertices in $A$.  So $T$ contains at most one vertex from $A$.
Note that $G$ is connected, for otherwise the vertex-set of the
component that contains $A$ would be a proper homogeneous set and not
a stable set.  So there is a shortest path $P$ from $A$ to $T$.  Let
$P=p_1$-$\cdots$-$p_k$, with $p_1\in A$, $p_2, \ldots, p_k\in
V(G)\setminus A$, $p_k=u$, $k\ge 1$, and $v,w\notin A$.  We choose $T$
so as to minimize $k$.  If $k=1$, let $p_2=v$.  We can pick vertices
$a_i\in A_i$ for each $i\in\{1,...,\ell\}$ so that $p_2$ has a
neighbor in the set $\{a_1, ..., a_\ell\}$.  Since $\ell$ is odd, and
by (\ref{hhn}), we may assume that $p_2$ is adjacent to $a_1$ and not
to $a_2$ and $a_3$.  Suppose that $k=1$ (so $u=p_1=a_1$ and $p_2=v$).
Then $\{u,v,w,a_2,a_3\}$ induces a hammer or a house (depending on
$w,a_3$).  Now suppose that $k\ge 2$.  By the minimality of $k$, the
vertices $v,w$ have no neighbor in $\{p_1, ..., p_{k-1}\}$.  Suppose
that $k=2$.  If any of $v,w$ is adjacent to $a_2$, then both are, for
otherwise $\{v,w,u,p_1,a_2\}$ induces a house; but then $\{v,w,a_2\}$
is a triangle, so we should have $k=1$.  Hence $k\ge 3$.  But then
$\{v,w,u,p_{k-1},p_{k-2}\}$ induces a hammer.  This completes the
proof of the theorem.
\end{proof}

\no {\it Proof of Theorem~\ref{thm:pbamain}}.  Since the class of
perfect graphs can be recognized in polynomial time \cite{perf-rec},
and since the WVC problem can be solved in polynomial time for perfect
graphs \cite{GLS} and for $O_3$-free graphs \cite{ML}, the theorem
follows from Theorems~\ref{thm:ML-1} and~\ref{thm:pba2}.
\hfill{$\Box$}

\section{WVC for ($P_5$, bull)-free graphs}\label{sec:p5bull}

In this section, we show the following result.  This was mentioned as 
an open problem in \cite{CamHoa}.
\begin{theorem}\label{thm:pbumain}
The \textsc{Weighted Vertex Coloring} problem can be solved in
polynomial time in the class of ($P_5$, bull)-free graphs.
\end{theorem}

We first establish a structure theorem for the complement graph of a
($P_5$, bull)-free graph.  Note that the bull is a self-complementary
graph.
\begin{theorem}\label{thm:pbu2}
Let $G$ be any prime (house, bull)-free graph.  Then $G$ is either
$(P_5, C_5)$-free or triangle-free.
\end{theorem}
\begin{proof}
Let $G$ be a prime (house, bull)-free graph, and suppose that $G$
contains a $P_5$ or a $C_5$.  So there exist five non-empty and
pairwise disjoint subsets $A_1, \ldots, A_5$ of $V(G)$ such that the
following properties hold, with subscripts modulo~$5$:
\begin{itemize}
\item
For each $i\in\{1,2,3,4\}$, $A_i$ is complete to $A_{i+1}$.
\item
For each $i\in\{1,2,3,4,5\}$, $A_i$ is anticomplete to $A_{i+2}$.
\item
$A_5$ is either complete or anticomplete to $A_1$.
\end{itemize}
Note that if $A_5$ is complete to $A_1$ the five sets play symmetric
roles.  Let $A=A_1\cup\cdots\cup A_5$.  We choose these sets so that
$A$ is inclusionwise maximal.  Let $B$ be the set of vertices of
$V(G)\setminus A$ that are complete to $A$.  We first claim that:
\begin{equation}\label{no4}
\longbox{For any vertex $v\in V(G)\setminus (A\cup B)$ and any
$i\in\{1,\ldots, 5\}$, $v$ is anticomplete to at least one of $A_i$,
$A_{i+1}$, $A_{i+2}$, $A_{i+3}$.}
\end{equation}
Proof: For each $i\in\{1,...,5\}$, let $a_i$ be a neighbor of $v$ in
$A_i$ (if any) and let $z_i$ be a non-neighbor of $v$ in $A_i$ (if
any).  Suppose that $v$ has neighbors in four sets $A_i$, $A_{i+1}$,
$A_{i+2}$, $A_{i+3}$.  Up to symmetry we may assume that
$i\in\{1,3,4\}$.  If $i=1$, then $v$ is complete to $A_5$, for
otherwise $\{a_1,v,a_3,a_4,z_5\}$ induces either a house or a bull
(depending on the adjacency between $a_1$ and $z_5$).  If $i=3$, then
$v$ is complete to $A_2$, for otherwise $\{a_1,z_2,a_3,a_4,v\}$
induces a house.  If $i=4$, then $v$ is complete to $A_3$, for
otherwise $\{a_1,a_2,z_3,a_4,v\}$ induces a house.  In all cases $v$
is complete to $A_{i-1}$, so $v$ has neighbors in all five sets.
Repeating this argument with each $i$ we obtain that $v\in B$, a
contradiction.  Thus (\ref{no4}) holds.

\medskip

Now we claim that:
\begin{equation}\label{no2}
\longbox{For any vertex $v\in V(G)\setminus (A\cup B)$ and any
$i\in\{1,\ldots, 4\}$, $v$ is anticomplete to at least one of $A_i$
and $A_{i+1}$. Also, if $A_5$ is complete to $A_1$, then $v$ is
anticomplete to one of $A_1,A_5$.}
\end{equation}
Proof: For each $i\in\{1,...,5\}$, let $a_i$ be a neighbor of $v$ in
$A_i$ (if any) and let $z_i$ be a non-neighbor of $v$ in $A_i$ (if
any).  Suppose that $v$ has neighbors in two consecutive sets $A_i$
and $A_{i+1}$.  \\
First suppose that $A_5$ is complete to $A_1$.  Up to symmetry, we may
assume that $i=1$.  Then $v$ is complete to $A_5$ or to $A_3$, for
otherwise $\{z_5, a_1,v,a_2, z_3\}$ induces a bull.  By symmetry we
may assume that $v$ is complete to $A_3$; and it follows from
(\ref{no4}) that $v$ has no neighbor in $A_5\cup A_4$.  Moreover $v$
is complete to $A_1$, for otherwise $\{z_1,a_2,v,a_3,z_4\}$ induces a
bull.  But now the sets $A_1, A_2\cup\{v\},A_3,A_4,A_5$ contradict the
maximality of $A$.  \\
Therefore we may assume that $A_5$ is anticomplete to $A_1$.  Up to
symmetry we have $i\in\{1,2\}$.  Suppose that $i=1$.  Suppose that $v$
has a non-neighbor $z_3\in A_3$.  Then $v$ is anticomplete to $A_4$,
for otherwise $\{a_1,a_2,z_3,a_4,v\}$ induces a house; and $v$ is
anticomplete to $A_5$, for otherwise $\{z_3,a_2,a_1,v,a_5\}$ induces a
bull; and $v$ is complete to $A_2$, for otherwise either
$\{v,a_1,z_2,z_3,a_2\}$ induces a house (if $a_2z_2\notin E(G)$) or
$\{v,a_2,z_2,z_3,z_4\}$ induces a bull (if $a_2z_2\in E(G)$).  But now
the sets $A_1\cup\{v\}, A_2,A_3,A_4,A_5$ contradict the maximality of
$A$.  Hence $v$ is complete to $A_3$.  By (\ref{no4}), $v$ has no
neighbor in $A_4\cup A_5$.  Then $v$ is complete to $A_1$, for
otherwise $\{z_1,a_2,v,a_3,z_4\}$ induces a bull.  But now the sets
$A_1,A_2\cup\{v\}, A_3,A_4,A_5$ contradict the maximality of $A$.
Finally suppose that $i=2$.  By the preceding point (the case $i=1$)
we may assume that $v$ is anticomplete to $A_1$.  Then $v$ is complete
to $A_4$, for otherwise $\{z_1,a_2,v,a_3,z_4\}$ induces a bull.  By
(\ref{no4}), $v$ is anticomplete to $A_5$.  By symmetry, $v$ is
complete to $A_2$.  But now the sets $A_1,A_2, A_3\cup\{v\},A_4,A_5$
contradict the maximality of $A$.  Thus (\ref{no2}) holds.

\medskip

Now we claim that:
\begin{equation}\label{bem}
\mbox{$B=\emptyset.$}
\end{equation}
Proof: Suppose that $B\neq\emptyset$.  Let $H$ be the component of
$G\setminus B$ that contains~$A$.  Since $G$ is prime, $V(H)$ is not a
proper homogeneous set, which implies that there exist non-adjacent
vertices $b\in B$ and $x\in V(H)$.  By the definition of $H$ there is
a shortest path $p_1$-$\cdots$-$p_k$ in $H$ with $p_1\in A$ and
$p_k=x$, and we choose the pair $b,x$ so as to minimize $k$.  We have
$k\ge 2$ since $x\notin A$.  We can pick vertices $a_i\in A_i$ for
each $i\in\{1,\ldots,5\}$ so that $p_2$ has a neighbor in
$\{a_1,...,a_5\}$.  We choose three vertices $u,v,w\in\{a_1,...,a_5\}$
so that: (i) $uv$ is the only edge in $G[u,v,w]$, and (ii) $u$ is the
only neighbor of $p_2$ among them; indeed we can find $u,v,w$ as
follows.  If $A_5$ is complete to $A_1$, then by~(\ref{no2}) and
symmetry we may assume that $p_2$ is adjacent to $a_1$ and has no
neighbor in $\{a_2,a_4,a_5\}$, and we set $u=a_1$, $v=a_2$, $w=a_4$.
Suppose that $A_5$ is anticomplete to $A_1$.  If $p_2$ is adjacent to
$a_1$ or $a_2$, let $\{u,v\}=\{a_1,a_2\}$, and let $w$ be a
non-neighbor of $p_2$ in $\{a_4,a_5\}$ ($w$ exists by (\ref{no2})).
The case when $p_2$ is adjacent to $a_5$ or $a_4$ is symmetric.
Finally if the only neighbor of $p_2$ in $\{a_1,...,a_5\}$ is $a_3$,
then let $u=a_3$, $v=a_2$ and $w=a_5$.  In either case, we see that
$b$ is adjacent to $p_2$, for otherwise $\{p_2,u,v,b,w\}$ induces a
bull.  So $k\ge 3$.  By the minimality of $k$, the vertices
$p_3,...,p_k$ have no neighbor in $A$, and $b$ is adjacent to each of
$p_2,...,p_{k-1}$.  Then $\{p_k,p_{k-1},p_{k-2},b,w\}$ induces a bull,
a contradiction.  Thus (\ref{bem}) holds.

\medskip

Now we claim that:
\begin{equation}\label{ais}
\mbox{For each $i\in\{1,...,5\}$,  $A_i$ is a stable set.}
\end{equation}
Proof: Suppose, up to symmetry, that $A_i$ is not a stable set for
some $i\in\{1,2,3\}$.  So $G[A_i]$ has a component $H$ of size at
least $2$.  Since $G$ is prime, $V(H)$ is not a homogeneous set, so
there is a vertex $z\in V(G)\setminus V(H)$ and two vertices $x,y\in
V(H)$ such that $z$ is adjacent to $y$ and not to $x$, and since $H$
is connected we may choose $x$ and $y$ adjacent.  By the definition of
$H$ we have $z\notin A_i$.  Since $z$ is adjacent to $y$ and not to
$x$, we have $z\notin A\cup B$.  Pick any $a'\in A_{i+1}$ and $a''\in
A_{i+2}$.  By (\ref{no2}) and since $z$ has a neighbor in $A_i$, $z$
is not adjacent to $a'$.  Then $\{z,y,x,a',a''\}$ induces a bull or a
house (depending on the adjacency between $z$ and $a''$), a
contradiction.  Thus (\ref{ais}) holds.

\medskip

To finish the proof of the theorem, suppose on the contrary that $G$
contains a triangle $T=\{u,v,w\}$.  By~(\ref{ais}), the graph $G[A]$
is triangle-free.  Moreover, by~(\ref{no2}), no triangle of $G$ has
two vertices in $A$.  So $T$ contains at most one vertex from $A$.
Note that $G$ is connected, for otherwise the vertex-set of the
component that contains $A$ would be a proper homogeneous set.  So
there is a shortest path $P$ from $A$ to $T$.  Let $P=
p_1$-$\cdots$-$p_k$, with $p_1\in A$, $p_2, \ldots, p_k\in
V(G)\setminus A$, $p_k=u$, $k\ge 1$, and $v,w\notin A$.  We choose $T$
so as to minimize $k$.  We can pick vertices $a_i\in A_i$ for each
$i\in\{1,...,5\}$ so that, up to symmetry $p_1=a_i$ for some
$i\in\{1,2,3\}$.  Let $p_0=a_{i+1}$.  Let $U$ be the set of neighbor
of $u$, and let $H$ be the component of $G[U]$ that contains $v$ and
$w$.  Since $V(H)$ is not a homogeneous set, there are vertices
$x,y\in V(H)$ and $z\in V(G)\setminus V(H)$ such that $z$ is adjacent
to $y$ and not to $x$, and since $H$ is connected we may choose such
$x$ and $y$ adjacent.  By the definition of $H$, the vertex $z$ is not
adjacent to $u$.  If $x$ is adjacent to $p_{k-1}$, then either $k=1$
and (\ref{no2}) is violated (because $x$ is adjacent to $p_1$ and
$p_0$), or $k\ge 2$ and $\{p_{k-1},p_k,x\}$ is a triangle that
contradicts the minimality of $k$.  So $x$ is not adjacent to
$p_{k-1}$, and similarly $y$ is not adjacent to $p_{k-1}$.  But then
$\{z,y,x,u,p_{k-1}\}$ induces a bull or a house (depending on the
adjacency between $z$ and $p_{k-1}$), a contradiction.  This completes
the proof of the theorem.
\end{proof}

\no {\it Proof of Theorem~\ref{thm:pbumain}}.  Since the WVC problem
can be solved in polynomial time for ($P_5, C_5$, house)-free graphs
\cite{CHMW}, and for $O_3$-free graphs \cite{ML}, the theorem follows
from Theorems~\ref{thm:ML-1} and~\ref{thm:pbu2}.  \hfill{$\Box$}

\section{WVC for (fork, bull)-free graphs}

In this section, we prove the following result.
\begin{theorem}\label{thm:forkbull}
The \textsc{Weighted Vertex Coloring} problem can be solved in
polynomial time in the class of (fork, bull)-free graphs.
\end{theorem}

We need some intermediate results. 
The following lemma is from \cite{CS}.
\begin{lemma}[\cite{CS}]\label{lem:bfnh}
In a bull-free graph $G$, let $B=\{v_1,...,v_\ell\}$ be the vertex-set
of a hole of length $\ell\ge 6$, with edges $v_iv_{i+1}$ for all $i$
modulo~$\ell$.  Then for every vertex $x$ in $V(G)\setminus B$ the set
$N_B(x)$ is either a stable set, or equal to $B$, or equal to
$\{v_{i-1},v_i,v_{i+1}\}$ for some $i$, or equal to
$\{v_{i-1},v_i,v_{i+1}, v_{i+3}\}$ for some $i$ and in this last case
$\ell=6$.
\end{lemma}

A \emph{wheel} (resp.~a \emph{fan}) is a graph that consists of a hole
$H$ of length at least $6$ (resp.~a path $H$ on $6$ vertices) plus a
vertex that is complete to $V(H)$.  An \emph{umbrella} (resp.~a
\emph{parasol}) is a graph that consists of a hole $H$ on five
vertices (resp.~a path $H$ on five vertices) plus a sixth vertex that
is complete to $V(H)$, and a seventh vertex that is adjacent to the
sixth vertex only.  In a wheel (resp.~fan, umbrella, parasol) the hole
or path $H$ is called the \emph{rim}.  The following lemma summarizes
results from \cite{DMP,KM,MPa,RS}.
\begin{lemma}\label{lem:homosets}
Let $G$ be a bull-free graph that contains as an induced subgraph
either a wheel, or an umbrella, or a parasol, or a fan.  Then $G$ has
a proper homogeneous set that contains the rim of this subgraph.
\end{lemma}

\begin{theorem}\label{thm:fbh}
Let $G$ be a prime (fork, bull)-free graph that contains a hole of
length $\ell\ge 6$.  Then $G$ is either a hole of length $\ell$ or a
bipartite graph.
\end{theorem}
\begin{proof}
Let $B=\{v_1,...,v_\ell\}$ be the vertex-set of such a cycle, with
edges $v_iv_{i+1}$ for all $i$ modulo~$\ell$.
For any $J\subseteq \{1,...,\ell\}$, let $S_J=\{x\in V(G)\setminus
B\mid$ $N(x)\cap B=\{v_j\mid j\in J\}\}$ (and we will write, for
example, $S_{123}$ instead of $S_{\{1,2,3\}}$).  Let
$A=S_{1,2,...,\ell}$, and $T=\bigcup_{i=1}^{\ell} S_{i-1,i,i+1}$, and
$F=S_{\emptyset}$.  We claim that:
\begin{equation}\label{7bafs}
\longbox{$V(G) = B \cup A \cup F\cup T\cup S_{135}\cup S_{246}$.  
Moreover, if $S_{135}\cup S_{246}\neq\emptyset$, then $\ell=6$.}
\end{equation}
Proof: Consider any vertex $x\in V(G)\setminus (A\cup F)$, and let
$X=N_B(x)$.  So $\emptyset\neq X\neq B$.  Since $G$ is bull-free,
Lemma~\ref{lem:bfnh} implies that $X$ is either (i) a stable set, or
(ii) equal to $\{v_{i-1},v_i, v_{i+1}\}$ for some $i$, or (iii) equal
to $\{v_{i-1},v_i, v_{i+1}, v_{i+3}\}$ for some $i$ and $\ell=6$.
Suppose that (i) holds, and say $v_1\in X$.  Then $v_\ell, v_2\notin
X$ since $X$ is stable, and $v_3\in X$, for otherwise
$\{v_3,v_2,v_1,v_\ell,x\}$ induces a fork.  Repeating this argument,
we see that $x$ is adjacent to every second vertex of $B$, which
implies that $\ell$ is even.  Moreover, if $\ell\ge 8$, then
$\{v_5,x,v_1,v_2, v_\ell\}$ induces a fork.  So $\ell=6$, and $x\in
S_{135}\cup S_{246}$, and the second sentence of (\ref{7bafs}) holds.
If (ii) holds, then $x\in T$.  Finally suppose that (iii) holds.  Then
$\{v_i,x,v_{i+3},v_{i+2}, v_{i+4}\}$ induces a fork.  Thus
(\ref{7bafs}) holds.  In particular, every vertex in $V(G)\setminus B$
has either zero, three or $\ell$ neighbors in $B$.

Next:
\begin{equation}\label{7a0}
A=\emptyset.
\end{equation}
Proof: In the opposite case the union of $B$ and any vertex in $A$
induces a wheel, which, by Lemma~\ref{lem:homosets}, contradicts the
fact that $G$ is prime.  So (\ref{7a0}) holds.

Next:
\begin{equation}\label{7s0}
\mbox{$T=\emptyset$.}
\end{equation}
Proof: Suppose that there is a vertex $u\in S_{123}$.  Let $Y$ be the
set of vertices that are complete to $\{v_1,v_3\}$ and anticomplete to
$B\setminus \{v_1,v_2,v_3\}$.  So $v_2, u\in Y$.  Let $Z$ be the
vertex-set of the component of $G[Y]$ that contains $v_2$ and $u$.
Since $G$ is prime, $Z$ is not a homogeneous set, so there are
vertices $y,z\in Z$ and a vertex $w\in V(G)\setminus Z$ that is
adjacent to $y$ and not to $z$.  Let $B_y$ be the vertex-set of the
hole induced by $(B\setminus \{v_2\}) \cup\{y\}$ and let $B_z$ be
defined similarly.  Then there is an integer $p$ such that $w$ has $p$
neighbors in $B_y$ and $p-1$ neighbors in $B_z$, which contradicts the
analogue of (\ref{7bafs}) applied to $B_y$ and $B_z$.  So (\ref{7s0})
holds.

Next:
\begin{equation}\label{7f0}
F=\emptyset.
\end{equation}
Proof: Suppose the contrary.  Let $f \in F$.  Since $G$ is prime it is
connected, so there is an edge $fu$ for some $u\in V(G)\setminus F$.
By (\ref{7a0}) and (\ref{7s0}), we have $u\in S_{135}\cup S_{246}$ and
$\ell=6$, say $u\in S_{135}$.  But then $\{f,u,v_1,v_2,v_6\}$ induces
a fork.  So, (\ref{7f0}) holds.

\medskip

Now if $S_{135}\cup S_{246}=\emptyset$, then Claims
(\ref{7bafs})--(\ref{7f0}) imply that $V(G)=B$, so $G$ is a hole of
length $\ell$.  Therefore let us assume that $S_{135}\cup
S_{246}\neq\emptyset$, and so $\ell=6$.  We claim that:
\begin{equation}\label{7sss}
\mbox{$S_{135}$ and $S_{246}$ are stable sets.}
\end{equation}
Proof: Suppose that $S_{135}$ is not a stable set.  Let $Y$ be the
vertex-set of a component of $G[S_{135}]$ of size at least~$2$.  Since
$G$ is prime, $Y$ is not a homogeneous set, so there are vertices
$y,z\in Y$ and a vertex $w\in V(G)\setminus Y$ that is adjacent to $y$
and not to $z$, and since $Y$ is connected we may choose $y$ and $z$
adjacent.  By (\ref{7bafs}), (\ref{7a0}), (\ref{7s0}) and (\ref{7f0})
we have $w\in S_{246}$.  But then $\{z,y,w,v_2,v_4\}$ induces a fork.
So (\ref{7sss}) holds.

By (\ref{7sss}), $V(G)$ can be partitioned into the two stable sets
$S_{135}\cup\{v_2,v_4,v_6\}$ and $S_{246}\cup\{v_1,v_3,v_5\}$, so $G$
is a bipartite graph.  This completes the proof of the theorem.
\end{proof}

By the preceding theorem, the WVC problem in case the graph contains a
hole of length at least~$6$ can be reduced to the same problem in a
graph that is either bipartite or an odd hole.  If the graph is
bipartite it is perfect, so we can use the algorithm from \cite{GLS}.
If it is an odd hole, the details can be worked out directly, as
explained in the following lemma.  A \emph{hyperhole} is any graph $H$
such that $V(H)$ can be partitioned into $\ell$ cliques $A_1, \ldots,
A_\ell$ (for some integer $\ell\ge 4$) such that for each $i$ modulo
$\ell$ the set $A_i$ is complete to $A_{i-1}\cup A_{i+1}$ and
anticomplete to $A_{i+2}\cup A_{i+3}\cup\cdots\cup A_{i-3}\cup
A_{i-2}$.  The WVC problem on a hole of length $\ell$, where the
$i$-th vertex has integer weight $w_i$, is equivalent to coloring a
hyperhole of length $\ell$ where the $i$-th set $A_i$ has size $w_i$.
\begin{lemma}
Let $H$ be a hyperhole, where $V(H)$ is partitioned into sets $A_1,
..., A_\ell$ as above, with $\ell$ odd, $\ell\ge 5$.  Then $\chi (G) =
\max \{\omega(H), \lceil\frac{2|V(H)|}{\ell-1}\rceil\}$.
\end{lemma}
\begin{proof}
Let $m=\frac{\ell-1}{2}$, and let $r(H)=\max \{\omega(H),
\lceil\frac{|V(H)|}{m}\rceil\}$.  Clearly we have $\chi(H)\ge
\omega(H)$, and also $\chi(H)\ge \lceil\frac{|V(H)|}{m}\rceil$ since
$m$ is the maximum size of a stable set in $H$.  So $\chi(H)\ge r(H)$.
Now let us show that $H$ admits a coloring of size $r(H)$.  We prove
this by induction on $V(H)$.  Note that the maximal cliques of $H$ are
the sets $A_i\cup A_{i+1}$, $i=1,...,\ell$.

First suppose that there is an integer $i$ such that $|A_i\cup
A_{i+1}| <\omega(G)$, say $i=\ell$.  For each even $j\in
\{2,4,...,\ell-1\}$ pick a vertex $a_j\in A_j$, and let $S=\{a_2, a_4,
..., a_{\ell-1}\}$.  Then $S$ is a stable set, and $S$ meets every
$\omega(H)$-clique of $H$, so $\omega(H\setminus S)=\omega(H)-1$.
Moreover $|S|=m$, so $\lceil\frac{|V(H\setminus S)|}{m}\rceil =
\lceil\frac{|V(H)|}{m}\rceil -1$.  It follows that $r(H\setminus
S)=r(H)-1$.  If $A_j=\{a_j\}$ for some even $j$, then $H\setminus S$
is a perfect graph (indeed a chordal graph), so $\chi(H\setminus
S)=\omega(H\setminus S)= r(H\setminus S)$.  Otherwise $H\setminus S$
is a hyperhole of length $\ell$, and, by the induction hypothesis,
$\chi(H\setminus S)=r(H\setminus S)$.  In either case, we can take any
$\chi(H\setminus S)$-coloring of $H\setminus S$ and add $S$ as a new
color class, and we obtain a coloring of $H$ with $r(H)$ colors.

Now suppose that $|A_i\cup A_{i+1}| =\omega(G)$ for all $i$.  Since
$\ell$ is odd, this means that the numbers $|A_i|$ ($i\in
\{1,...,\ell\}$) are all equal to some integer $q$.  Then $\omega(H)=
2q$ and $|V(H)|=\ell q=(2m+1)q$, so $\omega(H) <
\lceil\frac{|V(H)|}{m}\rceil$, and so $r(H)=
\lceil\frac{|V(H)|}{m}\rceil$.  Let $S$ be defined as above.  Then, as
above, we have $\lceil\frac{|V(H\setminus S)|}{m}\rceil =
\lceil\frac{|V(H)|}{m}\rceil -1$.  Hence $r(H\setminus S)=r(H)-1$, and
again we can take any $\chi(H\setminus S)$-coloring of $H\setminus S$
and add $S$ as a new color class, and we obtain a coloring of $H$ with
$r(H)$ colors.
\end{proof}

\begin{theorem}\label{thm:fbp}
Let $G$ be a prime (fork, bull)-free graph that contains no hole of
length at least $6$.  Suppose that $G$ contains a $P_5$
$v_1$-$v_2$-$v_3$-$v_4$-$v_5$.  Then either $G$ has a clique cutset,
or there is an optimal coloring of $G$ in which $\{v_1,v_3,v_5\}$ is a
color class.
\end{theorem}
\begin{proof}
Let $B=\{v_1,...,v_5\}$.  For any $J\subseteq \{1,...,5\}$, let
$S_J=\{x\in V(G)\setminus B\mid$ $N(x)\cap B=\{v_j\mid j\in J\}\}$.
Let $A=S_{1,2,3,4,5}$ and $F=S_{\emptyset}$.  We claim that:
\begin{equation}\label{pvg}
\longbox{$V(G) = B \cup A\cup F\cup S_1\cup S_5 \cup S_{12}\cup
S_{45}\cup S_{24}\cup S_{123}\cup S_{234}\cup S_{345}\cup S_{135}\cup
S_{1345}\cup S_{1235}$.}
\end{equation}
Proof: Consider any vertex $x\in V(G)\setminus (A\cup F)$, and let
$X=N_B(x)$; so $1\le |X|\le 4$.  \\
Suppose that $|X|=1$, so $X=\{v_i\}$ for some $i$.  If $i\in\{2,3\}$,
then $\{v_{i+2}, v_{i+1}, v_i, v_{i-1}, x\}$ induces a fork.  Hence,
by symmetry, $i\in\{1,5\}$ and so $x\in S_1\cup S_5$.  \\
Now suppose that $|X|=2$.  Up to symmetry we have the following six
cases.  (i)~$X=\{v_1,v_2\}$.  Then $x\in S_{12}$.
(ii)~$X=\{v_1,v_3\}$.  Then $\{v_5,v_4,v_3,v_2,x\}$ induces a fork.
(iii)~$X=\{v_1,v_4\}$.  Then $\{v_1,x,v_4,v_3,v_5\}$ induces a fork.
(iv)~$X=\{v_1,v_5\}$.  Then $B\cup\{x\}$ induces a hole of length~$6$,
a contradiction.  (v)~$X=\{v_2,v_3\}$.  Then $\{v_1,v_2,x,v_3,v_4\}$
induces a bull.  (vi)~$X=\{v_2,v_4\}$.  Then $x\in S_{24}$.  \\
Now suppose that $|X|=3$.  Up to symmetry we have the following four
cases.  (i)~$X=\{v_i,v_{i+1},v_{i+2}\}$ for some $i\in \{1,2,3\}$.
Then $x\in S_{123}\cup S_{234}\cup S_{345}$.  (ii)~$X=\{v_1,v_2,v_i\}$
for some $i\in\{4,5\}$.  If $i=4$, then $\{v_1,x,v_4,v_3,v_5\}$
induces a fork.  If $i=5$, then $\{v_3,v_2,v_1,x,v_5\}$ induces a
bull.  (iii)~$X=\{v_2,v_3,v_5\}$.  Then $\{v_1,v_2,x,v_3,v_4\}$
induces a bull.  (iv)~$X=\{v_1,v_3,v_5\}$.  Then $x\in S_{135}$.  \\
Finally suppose that $|X|=4$, and let $v_j$ be the non-neighbor of $x$
in $B$.  If $j=1$ or $3$, then $\{v_1,v_2,v_3,x,v_5\}$ induces a bull.
Similarly there is a bull if $j=5$.  So $j\in \{2,4\}$, and so $x\in
S_{1345}\cup S_{1235}$.  Thus (\ref{pvg}) holds.

\medskip

Define the following sets. Let:
\begin{eqnarray*}
V_1 &:=& \{v_1\}\cup S_{12}. \\
V_2 &:=& \{v_2\}\cup S_{123}.\\
V_3 &:=& \{v_3\}\cup S_{24}\cup S_{234}.\\
V_4 &:=& \{v_4\}\cup S_{345}.\\
V_5 &:=& \{v_5\}\cup S_{345}.\\
A' &:=&A\cup S_{135}\cup S_{1235}\cup S_{1345}.
\end{eqnarray*}

Now we claim that:
\begin{equation}\label{vihs}
\mbox{$V_1$, $V_3$ and $V_5$ are homogeneous sets.}
\end{equation}
Proof: Suppose that $V_1$ is not a homogeneous set.  So there are
vertices $y,z\in V_1$ and a vertex $w\in V(G)\setminus V_1$ that is
adjacent to $y$ and not to $z$.  Clearly $w\notin\{v_2,v_3,v_4,v_5\}$.
Then $yz\in E(G)$, for otherwise $\{v_4,v_3,v_2,y,z\}$ induces a fork.
Suppose that $wv_2\notin E(G)$.  Then $wv_3\in E(G)$, for otherwise
$\{w,y,z,v_2,v_3\}$ induces a bull; and $wv_5\notin E(G)$, for
otherwise $\{z,y,w,v_3,v_5\}$ induces a fork; and $wv_4\notin E(G)$,
for otherwise $\{v_2,v_3,w,v_4,v_5\}$ induces a bull.  But then
$\{v_5,v_4,v_3,v_2,w\}$ induces a fork.  Hence $wv_2\in E(G)$.
Suppose that $wv_3\in E(G)$.  Then $wv_4\in E(G)$, for otherwise
$\{z,v_2,w,v_3,v_4\}$ induces a bull; and $wv_5\in E(G)$, for
otherwise $\{y,w,v_3,v_4,v_5\}$ induces a bull.  But then
$\{z,y,w,v_3,v_5\}$ induces a fork.  Hence $wv_3\notin E(G)$.  Then,
by (\ref{pvg}), $w \in S_{24}$.  But, then $\{y, w, v_3, v_4, v_5\}$
induces a fork, a contradiction.  So $V_1$ is a homogeneous set, and
similarly $V_5$ is a homogeneous set.  \\
Now suppose that $V_3$ is not a homogeneous set.  So there are
vertices $y,zi\in V_3$ and a vertex $w\in V(G)\setminus V_3$ that is
adjacent to $y$ and not to $z$.  Clearly $w\notin\{v_1,v_2,v_4,v_5\}$,
and $w \notin S_{24}\cup S_{234}$.  The vertex $w$ has a neighbor in
$\{v_1,v_2\}$, for otherwise $\{v_1,v_2,w,y,z\}$ induces a fork or a
bull (depending on the adjacency between $y$ and $z$), and similarly
$w$ has a neighbor in $\{v_4,v_5\}$.  So, by (\ref{pvg}), $w \in
S_{135} \cup S_{1345} \cup S_{1235}$.  If $w \in S_{135}$, then
$\{v_1,v_2,z,v_4,v_5,w\}$ induces a hole of length~$6$, and if $w \in
S_{1235} \cup S_{1345}$, then $\{v_1, v_2, y, z, v_5\}$ or $\{v_4,
v_5, y, z, v_1\}$ induces a bull.  Thus (\ref{vihs}) holds.

\begin{equation}\label{v2v4}
\mbox{$V_2$ is anticomplete to $V_4$.}
\end{equation}
Proof: Suppose that there is an edge $st$ with $s\in V_2$ and $t\in
V_4$.  The definition of these sets implies that $s\in S_{123}$ and
$t\in S_{345}$.  Then $\{v_1,s,v_3,t,v_5\}$ induces a bull, a
contradiction.  Thus (\ref{v2v4}) holds.

\medskip

It follows from (\ref{vihs}) and (\ref{v2v4}) that for each
$i\in\{1,2,3,4\}$ the set $V_i$ is complete to $V_{i+1}$ and that
there is no other edge between any two $V_i$'s.  Moreover, since $G$
is quasi-prime, each of $V_1, V_3, V_5$ is a clique.  (Also, since
$V_3$ is a clique, we have $S_{24}=\emptyset$.)  Claim~(\ref{vihs})
also implies that $V_1\cup V_3\cup V_5$ is complete to $A'$.

Now we claim that:
\begin{equation}\label{psf1}
\longbox{$S_1$ is anticomplete to $V_2\cup V_3\cup V_4\cup V_5\cup
A'$.  Similarly, $S_5$ is anticomplete to $V_1\cup V_2\cup V_3\cup
V_4\cup A'$.}
\end{equation}
Proof: Suppose, up to symmetry, that there is an edge $st$ with $s\in
S_1$ and $t\in V_2\cup V_3\cup V_4\cup V_5\cup A'$.  If $t\in V_3\cup
V_5$, then (\ref{vihs}) is contradicted since $sv_3, sv_5\notin E(G)$.
If $t\in V_4$, then $t\in S_{345}$, and $\{s,t,v_4,v_3,v_2\}$ induces
a bull.  So $\{v_1, v_2\} \subseteq N_B(t)$.  Since $\{v_2, v_3, v_4,
s, t\}$ does not induce a fork or a bull, and by~(\ref{pvg}), we have
$N_B(t) =B$.  But then $B\cup\{s,t\}$ induces a $6$-fan, which
contradicts Lemma~\ref{lem:homosets}.  Thus (\ref{psf1}) holds.

\begin{equation}\label{psf2}
\mbox{$F$ is anticomplete $V_1\cup V_2\cup V_3\cup V_4\cup V_5\cup A'$.}
\end{equation}
Proof: Suppose on the contrary that there is an edge $ft$ with $f\in
F$ and $t\in V_1\cup V_2\cup V_3\cup V_4\cup V_5\cup A'$.  If $t\in
V_1\cup V_3\cup V_5$, then (\ref{vihs}) is contradicted since $f$ has
no neighbor in $B$.  If $t\in S_{123}\cup S_{345} \cup S_{1235}\cup
S_{1345}$, then there is $i\in\{1,2,3\}$ such that
$\{f,t,v_i,v_{i+1},v_{i+2}\}$ induces a bull.  If $t\in S_{135}$, then
$\{f,t,v_3,v_2,v_4\}$ induces a fork.  If $t\in A$, then
$B\cup\{t,f\}$ induces a parasol, which contradicts
Lemma~\ref{lem:homosets}.  Thus (\ref{psf2}) holds.

\begin{equation}\label{psf3}
\mbox{No component of $G[F\cup S_1\cup S_5]$ contains vertices from
both $S_1,S_5$.}
\end{equation}
Proof: In the opposite case, there is a chordless path $P$ with an end
in $S_1$, an end in $S_5$, and interior in $F$.  Then $V(P)\cup B$
induces a hole of length at least~$7$, a contradiction.  Thus
(\ref{psf3}) holds.

\begin{equation}\label{psf4}
\mbox{Any component $H$ of $G[S_1\cup S_5\cup F]$ satisfies either
$N(V(H))\subseteq V_1$ or $N(V(H))\subseteq V_5$.}
\end{equation}
Proof: This follows immediately from (\ref{psf1}), (\ref{psf2}) and (\ref{psf3}).

\medskip

Suppose that $S_1\cup S_5\cup F\neq\emptyset$.  So there exists a
component $H$ of $G[S_1\cup S_5\cup F]$ such that, by
(\ref{psf1})--(\ref{psf4}) and up to symmetry, $N(V(H))\subseteq V_1$.
So $V_1$ is a clique cutset of $G$, and in this case the theorem
holds.
Therefore we may assume that $S_1\cup S_5\cup F=\emptyset$.  It
follows that $V(G)=V_1\cup V_2\cup V_3\cup V_4\cup V_5\cup A'$.
Finally, we claim that:
\begin{equation}\label{vvvc}
\mbox{There is an optimal coloring of $G$ in which $\{v_1,v_3,v_5\}$
is a color class.}
\end{equation}
Proof: Let $f$ be any optimal coloring of $G$, and let $c$ be the
color of $v_3$.  If $c$ appears in $V_1$, then, up to swapping
vertices of $V_1$, we may assume that $c=f(v_1)$.  If $c$ does not
appear in $V_1$, then, since $N(V_1)=V_2\cup A' \subset N(v_3)$, color
$c$ does not appear in $V_2\cup A'$.  Hence we can safely change the
color of $v_1$ to $c$.  We can do the same in $V_5$ with $v_5$.  Thus
we obtain a coloring of $G$ where $\{v_1,v_3,v_5\}$ is a color class,
without increasing the number of colors.  So (\ref{vvvc}) holds.  This
completes the proof of the theorem.
\end{proof}

\no{\it Proof of Theorem~\ref{thm:forkbull}.} Here is an algorithm for
coloring a (fork, bull)-free graph $G$.  Using modular decomposition
(Theorem~\ref{thm:ML-1}), we may assume that $G$ is prime; and using
clique cutset decomposition (Theorem~\ref{thm:ML-2}), we may assume
that $G$ has no clique cutset.  If $G$ contains a hole of length at
least~$6$ (and this can be tested in polynomial time), then by
Theorem~\ref{thm:fbh}, $G$ is either an odd hole or a bipartite graph,
and in the case the solution can be computed directly, as explained
above.  Suppose that $G$ contains no hole of length at least~$6$ and
that $G$ contains a $P_5$ (and this can be tested in time
$O(|V(G)|^5)$).  By Theorem~\ref{thm:fbp}, we find a stable set $S$
(of size~$3$) such that $\chi(G)=\chi(G\setminus S)+1$, and we apply
the algorithm recursively on $G\setminus S$.  Finally, suppose that
$G$ contains no $P_5$.  Then we can use the algorithm for the class of
($P_5$, bull)-free graphs given in Section~\ref{sec:p5bull}.  Clearly
the total complexity is polynomial.  \hfill $\Box$
 
\subsection*{Concluding remarks} 

In this paper, we studied the computational complexity of
\textsc{Weighted Vertex Coloring} in classes of graphs defined by two
forbidden induced subgraphs, in particular for the class of $(H_1,
H_2)$-free graphs where $H_1, H_2$ are connected graphs on five
vertices.  The results of this paper together with earlier known
results (see \cite{Malyshev2014, ML, DP2018}) imply that the WVC
problem is solvable in polynomial time for all but three classes of
graphs defined by two forbidden connected induced subgraphs on five
vertices.  These three classes are $(P_5, \overline{2P_2+P_1})$-free
graphs, $(P_5, K_{2,3})$-free graphs, and $(P_5,
\overline{K_3+O_2})$-free graphs, and for each of them the complexity
status is still unknown.  We conjecture that for these classes too the
WVC problem is solvable in polynomial time.  Moreover, we refer to
\cite{DP2018} for more open problems on ($P_5, H$)-free graphs, for
various $H$.

\subsection*{Acknowledgement} 

The first author would like to thank Mathew~C.~Francis for fruitful
discussions.
The second and third authors are
supported by ANR grant  ANR-13-BS02-0007 - STINT.


\end{document}